\documentclass[12pt]{article}

\usepackage[utf8]{inputenc}
\usepackage[a4paper,nomarginpar]{geometry}
\usepackage[english]{babel}
\usepackage{enumitem}
\usepackage{amsmath,amsfonts,amssymb,amsthm,amscd}
\usepackage{tikz,xcolor}
\usepackage{mathrsfs}
\allowdisplaybreaks

\theoremstyle{plain}
\newtheorem{theorem}{Theorem}
\newtheorem{proposition}[theorem]{Proposition}
\newtheorem{lemma}[theorem]{Lemma}
\newtheorem{corollary}[theorem]{Corollary}

\theoremstyle{definition}
\newtheorem{definition}[theorem]{Definition}
\newtheorem{remark}[theorem]{Remark}
\newtheorem{example}[theorem]{Example}

\newcommand{\N}{\mathbb{N}} 
\newcommand{\Z}{\mathbb{Z}} 
\newcommand{\K}{\mathbb{K}} 

\newcommand{\card}{\operatorname{card}}

\newcommand{\ldens}{\operatorname{\underline{dens}}}
\newcommand{\udens}{\operatorname{\overline{dens}}}

\title{Distributional and mean Li--Yorke chaos for weighted shifts on Fréchet sequence spaces}

\author{  Jo\~ao V. A. Pinto}

\date{}

\begin{document}
	
	\maketitle
	
	\begin{abstract}
		In this paper, we give characterizations of distributional chaos and mean Li–Yorke chaos for weighted backward shifts acting on general Fréchet sequence spaces. As an application, we derive criteria for these two types of chaos in the setting of Köthe sequence spaces $\lambda_p(A,J)$ for $p \in \{0\}\cup [1, \infty)$ and $J=\N$ or $J=\Z$.
	\end{abstract}
	\bigskip\noindent
	{\bf Keywords:} distributional chaos, Fréchet sequence spaces, Köthe sequence spaces, mean Li–Yorke chaos.
	
	\bigskip\noindent
	{\bf 2020 Mathematics Subject Classification:} Primary 47A16, 47B33; Secondary 46E15, 46E30.
	\section{Introduction}
	Linear dynamics is a branch of mathematics at the intersection of dynamical systems and operator theory. It aims to study dynamical properties of continuous linear operators acting on topological vector spaces. For studies focused on properties related to hypercyclicity (the existence of a dense orbit), one of the most investigated notions in linear dynamics, such as mixing, weakly mixing, and Devaney’s chaos, we refer the reader to \cite{BaMa,GE-P}.
	
	Over the last decade, other notions of chaos—namely, distributional chaos, Li–Yorke chaos, and mean Li–Yorke chaos—focusing on the dynamics of pairs of points, have been extensively studied in the context of linear dynamics. In \cite{Nilson0}, the authors developed a general theory of distributional chaos in the linear setting of Fréchet spaces. Among other results, they established the Distributional Chaos Criterion, which will be employed in the present work. In \cite{Nilson1}, this criterion was used to obtain a complete characterization, in the form of an equivalence, for weighted backward shifts acting on the spaces $\ell^p(X)$ ($p \in [1, \infty)$) and $c_0(X)$, where $X=\N$ or $X=\Z$; these results appear as corollaries of more general theorems proved in \cite{Nilson1} for weighted composition operators. For studies of this notion of chaos in the setting of Banach spaces, we refer the reader to \cite{BBMP11, BBPW}. 
	
	In \cite{MOP}, sufficient conditions are given under which backward shifts on Köthe sequence spaces $\lambda_p(A,\mathbb{N})$ ($p \in [1, \infty)\cup \{0\}$) are distributionally chaotic. In the first part of this work, our aim is to provide a complete characterization of distributional chaos for weighted backward shifts in a setting more general than Köthe sequence spaces, namely, Fréchet sequence spaces. These spaces are subspaces of $\mathbb{K}^{\mathbb{N}}$ endowed with a topology that turns them into Fréchet spaces and ensures the continuity of the canonical projections. Our approach to obtaining such a characterization is based on the method used in \cite{Nilson1}, based on the Distributional Chaos Criterion. 
	
	In \cite{BBMP2}, the property of Li–Yorke chaos was studied in the linear context for operators acting on Fréchet spaces. Building on the framework established in \cite{BBMP2}, the authors in \cite{Nilson2} provide a full characterization of Li–Yorke chaos for weighted composition operators on the spaces $L^p(\mu)$ ($p \in [1, \infty)$) and $C_0(\Omega)$, as well as for weighted backward shifts on arbitrary Fréchet sequence spaces. 
	
	In order to mention other related works, we refer to \cite{Tlu, XWu}. In \cite{Tlu}, the authors study distributional chaos for weighted backward shifts on spaces of the form $\Sigma(X):=X^{\mathbb{N}}$, where $X$ is a Banach space. In \cite{XWu}, the authors provide a characterization of Li--Yorke chaos for (unweighted) shifts on Köthe sequence spaces $\lambda_p(A,\N)$, with $p \in [1,\infty] \cup \{0\}$.
	
	An important variant of Li--Yorke chaos is mean Li--Yorke chaos. In \cite{BBP,BBPW}, this notion was studied in the context of Banach spaces, and in \cite{Chi} it was generalized to complete metrizable topological groups, in particular to Fr\'echet spaces. In \cite{Nilson1}, a characterization of this property was obtained for weighted composition operators on the spaces $L^p(\mu)$ ($p \in [1, \infty)$) and $C_0(\Omega)$, and, as corollaries, for  weighted backward shifts on the spaces $\ell^p(X)$ ($p \in [1, \infty)$) and $c_0(X)$, where $X=\mathbb{N}$ or $X=\mathbb{Z}$. In the second part of the present work, we use the results developed in \cite{Chi} to provide a complete characterization of mean Li--Yorke chaos for weighted backward shifts on Fr\'echet sequence spaces satisfying the following natural condition:
	\begin{itemize}
		\item[\rm (C)] For each $n \in \mathbb{N}$, $m \in \mathbb{Z}$, and $x=(x_j)_{j \in \mathbb{Z}}$, we have
		\[
		|x_m|\,\|e_m\|_n \le \|x\|_n.
		\]
	\end{itemize}
	It is straightforward to verify that Köthe sequence spaces satisfy condition (C). 
	
	The paper is organized as follows. In Section 2, we recall some definitions of Fr\'echet sequence spaces and fix the notation. In Section 3, we establish a characterization of distributional chaos for weighted backward shifts in the more general setting of Fr\'echet sequence spaces; as a consequence, we obtain corollaries characterizing this property for Köthe sequence spaces. Finally, in Section 4, we characterize mean Li--Yorke chaos for weighted backward shifts on Fr\'echet sequence spaces satisfying condition (C).

	\section{Preliminaries}
	Throughout, $\mathbb{K}$ denotes either the field $\mathbb{R}$ of real numbers or the field $\mathbb{C}$ of complex numbers, $\mathbb{Z}$ denotes the ring of integers, $\mathbb{N}$ denotes the set of all positive integers, and $\mathbb{N}_0 = \mathbb{N} \cup \{0\}$. A vector space $X$ is said to be a {\it Fréchet space} if it is endowed with an increasing sequence $(\left\| \cdot \right\|_k)_{k \in \N}$ of seminorms (called a {\it fundamental sequence of seminorms}) that defines a metric 
	\begin{equation}\label{d}
		d(x,y) := \sum_{k=1}^{\infty} \frac{1}{2^k}\,\min\{1,\|x-y\|_k\}, \quad \text{for }x,y \in X,
	\end{equation}
	under which $X$ is complete.
	\begin{definition}
		A Fréchet space $X$ which is a vector subspace of the product space $\mathbb{K}^{\mathbb{N}}$ is a {\it Fréchet sequence space} if the inclusion map $X \to \mathbb{K}^{\mathbb{N}}$ is continuous, i.e., convergence in $X$ implies coordinatewise convergence.
	\end{definition}
	If $w := (w_n)_{n\in\mathbb{N}}$ is a sequence of nonzero scalars, the closed graph theorem implies that the \emph{unilateral weighted backward shift}
	\[
	B_w(x_1, x_2, x_3, \ldots) := (w_1 x_2,\, w_2 x_3,\, w_3 x_4, \ldots)
	\]
	
	is a continuous linear operator on $X$ provided that it maps $X$ into itself. If $w := (1)_{n\in\mathbb{N}}$, then we denote $B_w=B$.
	\begin{definition}
		Let $X$ be a Fréchet sequence space. The canonical vectors $e_n := (\delta_{n,j})_{j\in\mathbb{N}} \in \mathbb{K}^{\mathbb{N}}$ ($n\in\mathbb{N}$) form a {\it basis} of $X$ if they belong to $X$ and
		\[
		x = \sum_{n=1}^\infty x_n e_n, \qquad \text{for all } x := (x_n)_{n\in\mathbb{N}} \in X.
		\]
	\end{definition}
	In the case where $X$ is a Fréchet sequence space with basis $(e_n)_{n\in\mathbb{N}}$, we define the set $c_{00}(\mathbb{N})$ to be the subspace of all sequences with only finitely many nonzero coordinates.
	
	Similarly, we can define the above concepts in the bilateral case: 
	\begin{definition}
		A Fréchet space $X$ which is a vector subspace of the product space $\mathbb{K}^{\mathbb{Z}}$ is a {\it Fréchet sequence space over} $\Z$ if the inclusion map $X \to \mathbb{K}^{\mathbb{Z}}$ is continuous, i.e., convergence in $X$ implies coordinatewise convergence.
	\end{definition}
	As previously, if $w := (w_n)_{n\in\mathbb{Z}}$ is a sequence of nonzero scalars, then the \emph{bilateral weighted backward shift}
	\[
	B_w((x_n)_{n\in\mathbb{Z}}) := (w_n x_{n+1})_{n\in\mathbb{Z}}
	\]
	
	is a continuous linear operator on $X$ provided that it maps $X$ into itself. If $w := (1)_{n\in\mathbb{Z}}$, then we denote $B_w=B$.
	\begin{definition}
		Let $X$ be a Fréchet sequence space over $\Z$. The canonical vectors $e_n := (\delta_{n,j})_{j\in\mathbb{Z}} \in \mathbb{K}^{\mathbb{Z}}$ ($n\in\mathbb{Z}$) form a {\it basis} of $X$ if they belong to $X$, and
		\[
		x = \sum_{n=-\infty}^{\infty} x_n e_n \qquad \text{for all } x := (x_n)_{n\in\mathbb{Z}} \in X.
		\]
	\end{definition}
	In the case where $X$ is a Fréchet sequence space over $\Z$ with basis $(e_n)_{n\in\mathbb{Z}}$, we define the set $c_{00}(\mathbb{Z})$ to be the subspace of all sequences with only finitely many nonzero coordinates.
	
	To introduce the main class of examples of Fréchet sequence spaces, we need the following definition:
	\begin{definition}
		Let $J=\N$ or $J=\Z$. A matrix $A = (a_{j,k})_{j\in J,\; k\in\mathbb{N}}$ is called a {\it  K\"othe matrix} if it satisfies:
		\begin{itemize}
			\item[(i)] \(a_{j,k} \ge 0\) for all \(j\in J\) and \(k\in\mathbb{N}\);
			\item[(ii)] for each fixed \(j\in J\), $ a_{j,k} \le a_{j,k+1}$, for all $k \in \N$;
			\item[(iii)] for each \(j\in J\) there exists at least one \(k\in \N\) such that \(a_{j,k} > 0\).
		\end{itemize}
	\end{definition}
	\begin{definition}\label{def1}
		Let $J=\N$ or $J=\Z$. Consider $p\in \{0\}\cup [1,\infty)$ and  a Köthe matrix $A = (a_{j,k})_{j\in J,\; k\in\mathbb{N}}$. The {\it associated Köthe sequence space} (or simply {\it Köthe sequence space}) $\lambda_p(A,J)$ is the Fréchet sequence space defined as follows:
		\begin{itemize}
			\item If $p\in[1,\infty)$, then
			\[
			\lambda_p(A,J) := \Big\{ (x_j)_{j\in J}\in \K^J : \sum_{j\in J} |a_{j,k} x_j|^p < \infty \text{ for all } k\in\mathbb{N} \Big\},
			\]
			endowed with the seminorms
			\[
			\|x\|_k := \left( \sum_{j\in J} |a_{j,k} x_j|^p \right)^{1/p}, 
			\quad\text{ for }\; x=(x_j)_{j\in J}\in \lambda_p(A,J)\;\text{ and } k\in\mathbb{N}.
			\]
			
			\item If $p=0$, then
			\[
			\lambda_0(A,J) := \Big\{ (x_j)_{j\in J}\in \K^J : \lim_{j\in J,\,|j|\to\infty} a_{j,k} x_j = 0 \text{ for all } k\in\mathbb{N} \Big\},
			\]
			endowed with the seminorms
			\[
			\|x\|_k := \sup_{j\in J} |a_{j,k} x_j|,
			\quad \text{ for }\; x=(x_j)_{j\in J}\in \lambda_0(A,J)\;\text{ and } k\in\mathbb{N}.
			\]
		\end{itemize}
	\end{definition}
	It is straightforward to verify that the sequence $(e_n)_{n \in J}$ of canonical vectors in $\K^J$ is a basis for $\lambda_p(A,J)$ where $1 \le p < \infty$ or $p=0$ and $J=\N$ or $J=\Z$. It is well known that the weighted backward shift on $\lambda_p(A,J)$ is continuous if, and only if, for all $k \in \N$ there exists $m \in \N$ such that $a_{j,k}=0$ whenever $a_{j+1,m}=0$ $(j \in J)$ and
	$$\sup_{j \in J} \frac{a_{j,k}|w_j|}{a_{j+1,m}} < \infty.$$
	For a detailed discussion of Köthe sequence spaces, see reference \cite{Ko}.
	\begin{example}
		Let $J=\N$ or $J=\Z$ and let $\nu:=(v_n)_{n \in J}\subset (0,\infty)$. Consider the Köthe matrix
		$A=(a_{j,k})_{j\in J,\,k\in\mathbb{N}}$ defined by $a_{j,k}=v_j$ for all  $k\in\mathbb{N}$.
		\begin{itemize}
			\item[\rm (a)] If $p=0$, then the Köthe space $\lambda_0(A,J)$ coincides with the classical Banach space
			\[
			c_0(\nu,J)=\{(x_j)_{j\in J}\in\mathbb{K}^J : v_jx_j\to 0 \text{ as } |j|\to\infty\},\quad \text{with norm }\left\| (x_n)_{n\in J} \right\|=\sup_{n\in J}\left| v_nx_n \right|.
			\] 
			When $\nu = (1)_{n \in J}$ we denote $c_0(\nu,J)$ by $c_0(J)$.
			\item[\rm (b)] If $1\le p<\infty$, then $\lambda_p(A,J)$ coincides with the classical Banach space
			\[
			\ell^p(\nu,J)=\{(x_j)_{j\in J}\in\mathbb{K}^J : \sum_{j\in J}|v_jx_j|^p<\infty\}\quad \text{with norm }\left\| (x_n)_{n\in J} \right\|=\left( \sum_{j\in J}|v_jx_j|^p \right)^{\frac{1}{p}}.
			\] 
			When $\nu = (1)_{n \in J}$ we denote $\ell^p(\nu,J)$ by $\ell^p(J)$.
		\end{itemize}
	\end{example}
	\begin{example}
		Let $J = \mathbb{N}$ or $J=\mathbb{Z}$. If 
		$a_{j,k} := (|j|+1)^k$ for all $j \in J$ and $k \in \mathbb{N}$, we denote by
		\[
		s(J):=\lambda_1(A,J),
		\]
		the space of {\it rapidly decreasing sequences on} $J$, which is a classical example of a non-normable Fréchet sequence space.
	\end{example}
	\section{Distributional chaos}
	The term 'chaos' was first introduced into the mathematical literature by Li and Yorke \cite{16} in their investigation of the dynamics of interval maps. Later, Schweizer and Smítal \cite{23} introduced the following notion, which can be seen as a natural extension of the original Li--Yorke concept:
	\begin{definition}
		Given a metric space $M$, a map $f : M \to M$ is said to be {\em distributionally chaotic} if there exist
		an uncountable set $\Gamma \subset M$ and $\varepsilon > 0$ such that each pair $(x,y)$ of distinct points in $\Gamma$
		is a {\em distributionally chaotic pair for $f$}, in the sense that
		\[
		\liminf_{k\to \infty}\frac{\text{card}\left( \{n \in \left\{ 1, \cdots, k \right\} : d(f^n(x),f^n(y)) < \varepsilon\} \right)}{k}=0
		\]
		and
		\[
		\limsup_{k\to \infty}\frac{\text{card}\left( \{n \in \left\{ 1, \cdots, k \right\} : d(f^n(x),f^n(y)) < \tau \} \right)}{k}=1, \ \text{ for all } \tau > 0,
		\]
		where $\text{card}(A)$ denotes the cardinality of the set $A\subset \N.$
	\end{definition}
	We also use the following notation: let $A\subset \N$, then we define
	$$\udens(A):=\limsup_{N \to \infty}\frac{\text{card}(\left\{ 1, \cdots, N \right\}\cap A)}{N}\quad \text{and} \quad \ldens(A):=\liminf_{N \to \infty}\frac{\text{card}(\left\{ 1, \cdots, N \right\}\cap A)}{N}.$$
	To prove our main result in this section, we need the following definition and theorem from \cite{Nilson0}:
	\begin{definition}Let $T$ be a continuous linear operator on a Fréchet space $X$. We say that $T$ satisfies the \emph{Distributional Chaos Criterion} (DCC) if there exist sequences $(x_k), (y_k)$ in $X$ such that:
		\begin{itemize}
			\item[\rm (a)] There exists $A \subset \mathbb{N}$ with $\udens(A) = 1$ such that $\lim_{n \in A} T^n x_k = 0$ for all $k \in \N$.
			\item[\rm (b)] $y_k \in \overline{\operatorname{span}\{x_n : n \in \mathbb{N}\}}$, $\lim_{k \to \infty} y_k = 0$ and there exist $\varepsilon > 0$ and an increasing sequence $(N_k)_{k\in \N}$ in $\mathbb{N}$ such that
			\[
			\operatorname{card}\{1 \le j \le N_k : d(T^j y_k, 0) > \varepsilon\} \ge N_k(1 - k^{-1})
			\]
			for all $k \in \mathbb{N}$.
		\end{itemize}
	\end{definition}
	\begin{theorem}\cite[Theorem 12]{Nilson0} \label{T0}
		Let $T$ be a continuous linear operator on a Fréchet space $Y$. Then, the following statements are equivalent:
		\begin{itemize}
			\item[\rm (a)] $T$ is distributionally chaotic;
			\item[\rm (b)] $T$ admits a {\em distributionally irregular vector}, that is, 
			a vector $y \in Y$ for which there are $m \in \mathbb{N}$ and $A, B \subset \mathbb{N}$ with $\udens(A) = \udens(B) = 1$ such that
			\[
			\lim_{n \in A} T^n y = 0
			\quad \text{and} \quad
			\lim_{n \in B} \|T^n y\|_m = \infty.
			\]
			\item[\rm (c)]  $T$ satisfies the DCC.
		\end{itemize}
	\end{theorem}
	In the following result, we establish a characterization of distributional chaos for bilateral weighted backward shifts on Fréchet sequence spaces over $\Z$.
	\begin{theorem}\label{T1}
		Let $X$ be a Fréchet sequence space over $\mathbb{Z}$, endowed with an increasing sequence $(\left\| \cdot \right\|_n)_{n \in \N}$ of seminorms, in which the sequence 
		$(e_n)_{n \in \mathbb{Z}}$ of canonical vectors is a basis. Suppose that the bilateral weighted backward shift $B_w$, with nonzero weights $w := (w_n)_{n \in \mathbb{Z}}$, is well-defined and continuous on $X$. Then, $B_w$ is distributionally chaotic if and only if there exist $D \subset \N$ with $\udens(D)=1$ and $I \subset \Z$ such that the following conditions hold:
		\begin{itemize}
			\item[\rm (A)] For all $i \in I$, we have
			$$\lim_{n \in D} w_{i-n}\cdots w_{i-1}e_{i-n}=0.$$
			\item[\rm (B)] There exist $m \in \N$ and an increasing sequence $(N_k)_{k \in \N}$ of positive integers such that for each $k \in \N$ there are $r:=r(k) \in \N$, indices $i_{1,k}, \cdots, i_{r,k} \in I$ and scalars $b_{1,k}, \cdots, b_{r,k} \in \K \setminus\left\{ 0 \right\}$ with $\left\| \sum_{j=1}^{r} b_{j,k} e_{i_{j,k}} \right\|_{p(k)} \ne 0$ and
			$$\card\left\{ 1 \le n \le N_k :\frac{\left\| \sum_{j=1}^{r} b_{j,k} w_{i_{j,k}-n}\cdots w_{i_{j,k} -1}e_{i_{j,k}-n} \right\|_{m} }{\left\| \sum_{j=1}^{r} b_{j,k} e_{i_{j,k}} \right\|_{p(k)}}> k \right\}> (1- k^{-1}) N_k,$$
			where $p(k):= m$, if $1\le k\le m$, and $p(k):= k$, if $k>m$.
		\end{itemize}
		
	\end{theorem}
	
	\begin{proof}
		$(\Rightarrow)$ Since $B_w$ is distributionally chaotic, by Theorem \ref{T0}, $B_w$ admits a distributionally irregular vector $x:= (x_n)_{n \in \Z}$, that is, there exist sets $D,E \subset \N$ with $\overline{\text{dens}}(D)=\overline{\text{dens}}(E)=1$ and $m \in \N$ such that 
		\begin{equation}\label{eq1}
			\lim_{n \in D} (B_w)^{n}(x)=0 \quad \text{and} \quad \lim_{n \in E} \left\|(B_w)^{n}(x)  \right\|_m = \infty.
		\end{equation}
		It is easy to see that we can take $m\in \N$ in (\ref{eq1}) sufficiently large such that $\left\| x \right\|_m \neq 0$. Take $I:=\left\{ i \in \Z: x_i \ne 0 \right\}$. Let $V$ be a neighborhood of $0$ in $X$. Then, by the equicontinuity of the family of maps $y:= (y_n)_{n \in \Z} \in X \mapsto y_k e_k \in X $ ($k \in \Z$), there exists a neighborhood $U$ of 0 in $X$ such that 
		\begin{equation}\label{eq2}
			y:= (y_n)_{n \in \Z} \in U \quad \Rightarrow  \quad y_ke_k \in V, \quad \text{for all } k \in \Z.
		\end{equation}
		By (\ref{eq1}) and (\ref{eq2}), there exists $n_0 \in \N$ such that  
		$$n \in D\,\,\, \text{and} \,\,\, n \ge n_0 \quad \Rightarrow \quad (B_w)^n(x) \in U \quad \Rightarrow \quad x_{i}w_{i -n}\cdots w_{i-1}e_{i-n} \in V, $$
		for all $i \in I$. Therefore, item (A) holds. On the other hand, for each $k \in \N$, since 
		$$\overline{\text{dens}}\left\{ n \in \N : \left\| (B_w)^n(x)\right\|_m > k(\left\| x \right\|_{p(k)} +1)\right\}= \overline{\text{dens}}(E)=1,$$
		there exists $N_k \in \N$, such that 
		$$\text{card}\left\{ 1\le n \le N_k : \left\| (B_w)^n(x)\right\|_m > k(\left\| x \right\|_{p(k)} +1)\right\} > N_k\left( 1 -\frac{1}{2k}\right). $$  
		The positive integers $N_k$ can be chosen so that the sequence $(N_k)_{k \in \N}$ is increasing. Fix $k$ and define 
		$$J_k := \left\{ 1\le n \le N_k : \left\| (B_w)^n(x)\right\|_m > k(\left\| x \right\|_{p(k)} +1)\right\}.$$
		Since, for each $n \in J_k$
		$$k(\left\| x \right\|_{p(k)} +1) <\left\| \left( B_w \right)^n(x) \right\|_m = \lim_{N \to \infty} \left\|\sum_{i=-N}^{N} x_i w_{i-n}\cdots w_{i-1} e_{i-n} \right\|_m,$$
		then, there exists $N \in \N$ large enough such that the following inequalities hold
		$$\left\|\sum_{i=-N}^{N} x_i w_{i-n}\cdots w_{i-1} e_{i-n} \right\|_m >k (\left\| x \right\|_{p(k)} +1) > k\left\| \sum_{i=-N}^{N} x_i e_i \right\|_{p(k)}, \quad \text{for all }n \in J_k.$$
		Therefore, item (B) holds. \\ 
		$(\Leftarrow)$ For this implication, we will use the Distributional Chaos Criterion. By item (A), for each $i \in I$ we have that 
		$$\lim_{n \in D}(B_w)^n(e_i)=\lim_{n \in D} w_{i-n}\cdots w_{i-1}e_{i-n}=0.$$
		Therefore, the item (a) of DCC holds. Now, for each $k \in \N$ consider
		$$ y_k : =\frac{\sum_{j=1}^{r} b_{j,k} e_{i_{j,k}}}{k \left\| \sum_{j=1}^{r} b_{j,k} e_{i_{j,k}} \right\|_{p(k)}},$$
		where $r=r(k) \in \N$, $i_{1,k}, \cdots , i_{r,k} \in I$ and $b_{1,k}, \cdots, b_{r,k} \in \K \setminus\left\{ 0 \right\}$ come from (B). Fix $n \in \N$ and $\varepsilon>0$. Take $k_0 \in \N$ such that $k_0 > \max \{ m,n\}$ and $\frac{1}{k_0} < \varepsilon.$ Then, for all $k \ge k_0$, we have 
		$$\left\| y_k \right\|_n = \left\|\frac{\sum_{j=1}^{r} b_{j,k} e_{i_{j,k}}}{k\left\| \sum_{j=1}^{r} b_{j,k} e_{i_{j,k}} \right\|_k} \right\|_n \le \frac{1}{k} < \varepsilon.$$
		Therefore, $y_k \to 0$, when $k \to \infty$. For $m \in \N$ of item (B), there is $\delta > 0$ such that 
		\begin{equation}\label{eq3}
			\left\| x \right\|_m > 1 \quad \Rightarrow d(x, 0)> \delta.
		\end{equation}
		Then, by the definition of $y_k$, item (B) and (\ref{eq3}), we have 
		$$\text{card} \left\{ 1 \le n \le N_k :d((B_w)^n(y_k),0)> \delta \right\}> (1- k^{-1}) N_k.$$
		Therefore, the distributional chaos criterion holds.
	\end{proof}
	\begin{remark}
		The condition (A) of Theorem \ref{T1} was used in the $(\Leftarrow)$ part of the proof to ensure that
		\begin{equation}\label{eq8}
			\lim_{n\in D}(B_w)^n(e_i)=0, \quad \text{for all }i\in I,
		\end{equation}
		where $\udens(D)=1$, and thereby allow us to use the DCC. In the unilateral case, (\ref{eq8}) is trivially satisfied with $D= \N$ and $e_i$ for all $i \in \N$. Therefore, in this context, we obtain the following unilateral characterization:
	\end{remark}
	\begin{theorem}\label{theo2}
		Let $X$ be a Fréchet sequence space, endowed with an increasing sequence $(\left\| \cdot \right\|_n)_{n \in \N}$ of seminorms, in which the sequence 
		$(e_n)_{n \in \mathbb{N}}$ of canonical vectors is a basis. Suppose that the unilateral weighted backward shift $B_w$, with nonzero weights $w := (w_n)_{n \in \mathbb{N}}$, is well-defined and continuous on $X$. Then, $B_w$ is distributionally chaotic if and only if the following condition holds:
		\begin{itemize}
			\item There exist $m \in \N$ and an increasing sequence $(N_k)_{k \in \N}$ of positive integers, such that for each $k \in \N$ there are $r:=r(k) \in \N$, indices $i_{1,k}, \cdots, i_{r,k} \in \N$ and scalars $b_{1,k}, \cdots, b_{r,k} \in \K \setminus\left\{ 0 \right\}$ with $\left\| \sum_{j=1}^{r} b_{j,k} e_{i_{j,k}} \right\|_{p(k)}\ne 0$ and
			$$\card \left\{ 1 \le n \le N_k :\frac{\left\| \sum_{j=1}^{r} b_{j,k} w_{i_{j,k}-n}\cdots w_{i_{j,k} -1}e_{i_{j,k}-n} \right\|_{m} }{\left\| \sum_{j=1}^{r} b_{j,k} e_{i_{j,k}} \right\|_{p(k)}}> k \right\} > (1- k^{-1}) N_k,$$
			where $p(k):= m$, if $1\le k\le m$, and $p(k):= k$, if $k>m$. We consider $e_{k}= (0)_{j \in \N}$ and $w_k=0$, for $k<1.$
		\end{itemize}
	\end{theorem}
	As a first application, we prove that Theorem \ref{theo2} recovers \cite[Theorem 11]{MOP}, which establishes a sufficient condition in the case of $\ell^p(\nu, \mathbb{N})$ for $p\in [1, \infty)$ or $c_0(\nu, \N)$. To this end, we introduce the following notation:
	\[
	S_{i,j}(\alpha) = \{ k \in [i,j]\cap \N : a_k \ge \alpha \},
	\]
	for positive integers $i < j$ and a number $\alpha > 0$.
	\begin{corollary}
		Let $(\nu_n)_{n \in \N}\subset (0, \infty)$ and fix $p \in \{0\} \cup [1,\infty)$. Assume that the backward shift $B$ is well-defined and continuous on $\ell^p(\nu, \N)$ (or $c_0(\nu,\N)$, if $p=0$). If there exist a sequence $(\alpha_n)_{n \in \mathbb{N}} \subset (0,+\infty)$ and increasing functions
		$j_0, j_1 : \mathbb{N} \to \mathbb{N}$ such that $j_1(n) - j_0(n) \ge n$ for all $n \in \N$ and
		\begin{itemize}
			\item[\rm (i)] $\lim_{n \to \infty} \frac{a_{j_1(n)}}{\alpha_n} = 0,$
			\item[\rm (ii)] $\lim_{n \to \infty} 
			\frac{ \operatorname{card}\left(  S_{j_0(n),\, j_1(n)}(\alpha_{n}) \right)}
			{j_1(n) - j_0(n)} = 1,$ 
		\end{itemize}
		then $B$ is distributionally chaotic.
	\end{corollary}
	\begin{proof}
		Without loss of generality, we assume that the sequence $(j_1(n) - j_0(n))_{n\in \N}$ is strictly increasing. For each $k\in \N$ take $n_k\in \N$ such that  $ \frac{a_{j_1(n_k)}}{\alpha_{n_k}} <\frac{1}{2k}$ and $\frac{ \text{card}\left(  S_{j_0(n_k),\, j_1(n_k)}(\alpha_{n_k}) \right)-1}
		{j_1(n_k) - j_0(n_k)} > (1-k^{-1}).$ Define $N_k:=j_1(n_k) - j_0(n_k)$, $k \in \N$. Then 
		$$\text{card}\left\{ 1\le i \le N_k: \frac{\left\| \frac{e_{j_1(n_k)-i}}{\alpha_{n_k}} \right\|}{\left\| \frac{e_{j_1(n_k)}}{\alpha_{n_k}} \right\|}>k  \right\} \ge \text{card}\left(  S_{j_0(n_k),\, j_1(n_k)}(\alpha_{n_k}) \right) -1 > (1-k^{-1})N_k.$$
	\end{proof}
	As a consequence of Theorem \ref{T1}, we obtain the following corollary, which characterizes distributional chaos in the context of weighted shifts on Köthe sequence spaces. This corollary is obtained by directly applying Theorem \ref{T1} to the seminorms from Definition \ref{def1}.
	\begin{corollary}\label{cor}
		Consider a Köthe sequence space $X := \lambda_p(A,\mathbb{Z})$, 
		where $A := (a_{j,k})_{j \in \Z, k \in \N}$ is a Köthe matrix and 
		$p \in \{0\} \cup [1,\infty)$. Let $w := (w_n)_{n\in\mathbb{Z}}$ be a sequence of 
		nonzero scalars such that the bilateral weighted backward shift $B_w$ is a 
		well-defined and continuous operator on $X$. 
		\begin{itemize}
			\item[\rm (a)] If $p=0$, then $B_w$ is distributionally chaotic if and only if there exist $D \subset \N$ with $\udens(D)=1$ and $I \subset \Z$ such that the following conditions hold:
			\begin{itemize}
				\item[\rm (A1)] $\lim_{n\in D}a_{i-n,k}w_{i-n}\cdots w_{i-1}=0, \quad \text{for all }k \in \N\text{ and } i \in I.$
				
				\item[\rm (A2)] There exist $m \in \N$ and an increasing sequence $(N_k)_{k \in \N}$ of positive integers, such that for each $k \in \N$ there are $r:=r(k) \in \N$, indices $i_{1,k}, \cdots, i_{r,k} \in I$ and scalars $b_{1,k}, \cdots, b_{r,k} \in \K \setminus\left\{ 0 \right\}$ with $\max_{1\le j \le r}\left| a_{i_{j,k},p(k)}b_{j,k} \right| > 0$ and
				\begin{equation}\label{eqqq3}
					\card \left\{ 1 \le n \le N_k :\frac{\max_{1\le j\le r} \left| a_{i_{j,k}-n,m}b_{j,k} w_{i_{j,k}-n}\cdots w_{i_{j,k} -1} \right| }{\max_{1\le j \le r}\left| a_{i_{j,k},p(k)}b_{j,k} \right|}> k \right\}> (1- k^{-1}) N_k,
				\end{equation}
				where $p(k):= m$, if $1\le k\le m$, and $p(k):= k$, if $k>m$.
			\end{itemize}
			\item[\rm (b)] If $p \in [1, \infty)$, then $B_w$ is distributionally chaotic if and only if  there exist $D \subset \N$ with $\udens(D)=1$ and $I \subset \Z$ such that the following conditions hold: 
			\begin{itemize}
				\item[\rm (B1)] $\lim_{n\in D}a_{i-n,k}w_{i-n}\cdots w_{i-1}=0, \quad \text{for all }k \in \N\text{ and } i \in I.$
				\item[\rm (B2)] There exist $m \in \N$ and an increasing sequence $(N_k)_{k \in \N}$ of positive integers, such that for each $k \in \N$ there are $r:=r(k) \in \N$, indices $i_{1,k}, \cdots, i_{r,k} \in I$ and scalars $b_{1,k}, \cdots, b_{r,k} \in \K \setminus\left\{ 0 \right\}$ with $\sum_{j=1}^{r} \left| a_{i_{j,k},p(k)}b_{j,k} \right|^p > 0$ and
				\begin{equation}\label{eqq4}
					\card \left\{ 1 \le n \le N_k :\frac{\sum_{j=1}^{r} \left| a_{i_{j,k}-n,m}b_{j,k} w_{i_{j,k}-n}\cdots w_{i_{j,k} -1} \right|^p }{\sum_{j=1}^{r} \left| a_{i_{j,k},p(k)}b_{j,k} \right|^p}> k^p \right\} > (1- k^{-1}) N_k,
				\end{equation} 
				where $p(k):= m$, if $1\le k\le m$, and $p(k):= k$, if $k>m$.
			\end{itemize}
		\end{itemize}
	\end{corollary}
	\begin{remark}
		For a fixed \(p \in [1,\infty)\), recall that \(\ell^p(\Z)=\lambda_p(A,\Z)\), where \(A=(a_{j,k})_{j\in\Z,\,k\in\N}\) is given by \(a_{j,k}=1\) for all \(j\in\Z\) and \(k\in\N\). Likewise, \(c_0(\Z)=\lambda_0(A,\Z)\). Therefore, by Corollary~\ref{cor}, we obtain characterizations of distributional chaos for weighted shifts on \(c_0(\Z)\) and \(\ell^p(\Z)\) equivalent to those given in \cite{Nilson1}. A unilateral version of Corollary~\ref{cor} follows from Theorem~\ref{theo2}. We leave the details to the reader. 
	\end{remark}
	Recall that an operator $T:X \to X$ is said to be {\it hypercyclic} if there exists $x \in X$ whose orbit is dense, i.e., $\overline{\left\{ T^n(x):n\in \N_0 \right\}}= X$. By \cite[Theorem 4.13]{GE-P}, a weighted backward shift $B_{w'}$, with nonzero weights $w':=(w'_n)_{n \in \Z}$, is  hypercyclic if, and only if, there exists an increasing sequence $(n_k)_{k \in \N}$ of positive integers, such that for each $\ell \in \Z$ we have
	\begin{equation}\label{equa3}
		w'_{\ell - n_j} \cdots w'_{\ell - 1} e_{\ell - n_j} \to 0
		\quad \text{and} \quad
		\frac{e_{\ell + n_j}}{w'_{\ell} \cdots w'_{\ell + n_j - 1}} \to 0
		\quad \text{as } j \to \infty.
	\end{equation}
	Examples of shifts that are hypercyclic but not distributionally chaotic are already known; see, for instance, \cite[Theorem 7]{BaRu} (in fact, this is a more involved example: it is shown that the shift is frequently hypercyclic). In what follows, to illustrate the use of the necessary direction in Theorem \ref{T1} (in particular, Corollary \ref{cor}), we will give an example of a weighted backward shift on $s(\Z)$ that is  hypercyclic, but it is not distributionally chaotic.
	\begin{example}
		Consider the weighted backward shift $B_w$ on $s(\Z)$ with weights
		$$(w_j)_{j \in \Z}:=(\cdots,\underset{\text{Block }B_n}{\underbrace{\overset{n\text{ times}}{\overbrace{2^{(-1)^{n}},\cdots,2^{(-1)^{n}}}},\overset{n\text{ times}}{\overbrace{2^{(-1)^{n+1}},\cdots,2^{(-1)^{n+1}}}}}},\cdots,\underset{\text{Block }B_2}{\underbrace{2,2,\frac{1}{2},\frac{1}{2}}},\underset{\text{Block }B_1}{\underbrace{\frac{1}{2},2}},\underset{j \ge 0}{\underbrace{2,2,2,2,\cdots}}).$$
		For each $n\in \N$, denote by $I_n$ the set of indices that compose the Block $n$. For example, \(I_1=\{-2,-1\}\), \(I_2=\{-6,-5,-4,-3\}\), and so on. Moreover, we denote $-I_n:=\left\{ -j:j \in I_n \right\}$ ($n \in \N$). We need the following lemma to prove that $B_w$ is not distributionally chaotic:
		\begin{lemma} We have that
			$$  \ldens\left( \bigcup_{n \in \N}^{}(-I_{2n-1}) \right)>0.$$ 
		\end{lemma}
		\begin{proof}
			Denote by $\mathcal{A}:=\bigcup_{n \in \N}^{}(-I_{2n-1})$. Take $N\in \N$ with $N>2$. Then, there exists $n \in \N$ such that 
			\begin{equation}\label{equa1}
				(2n+1)(2n+2)=\sum_{j=1}^{2n+1}\card(I_j) \ge N \ge \sum_{j=1}^{2n-1}\card(I_j)= 2n(2n-1).
			\end{equation}
			Then,
			\begin{equation}\label{equa2}
				\text{card}\left( \mathcal{A} \cap \left\{ 1, \cdots , N \right\} \right) \ge \sum_{j=1}^{n}\underset{=2(2j-1)}{\underbrace{\text{card}(I_{2j-1})}}=2n^2.
			\end{equation}
			Thus, by (\ref{equa1}) and (\ref{equa2}), we have
			$$\frac{\text{card}\left( \mathcal{A} \cap \left\{ 1, \cdots , N \right\} \right)}{N}\ge \frac{2n^2}{(2n+1)(2n+2)}=\frac{1}{2 + \frac{3}{n} + \frac{1}{n^2}}> \frac{1}{6}.$$
			Therefore, $\ldens(\mathcal{A})\ge \frac{1}{6}.$
		\end{proof}
		Continuing with our example, we claim that $B_w$ is not distributionally chaotic. Indeed, for all $i\in\mathbb{Z}$, there does not exist a subset 
		$D\subset\mathbb{N}$ with $\udens(D)=1$ such that either condition (A1) or (B1) 
		of Corollary~\ref{cor} holds, since $\ldens(\mathcal{A})>0$ and there exists a 
		constant $C>0$ (depending on $i$) such that
		$$\left|a_{i-n,\ell} w_{i-n}\cdots w_{i-1} \right|\ge C, \quad \text{whenever } n\in (\mathcal{A}+i)\cap \mathbb{N}\text{ and } \ell \in \N,$$
		where $\mathcal{A}+i := \left\{ a+i: a \in \mathcal{A} \right\}$.  \\
		Now, consider the sequence $(n_k)_{k \in \N}$ defined by $n_k:= 2k(2k-1)+k$. Then, for each $i \in \Z$ there exists a constant $C_i >0$ such that  
		$$\left|a_{i-n_k,\ell} w_{i-n_k}\cdots w_i \right|\le C_i\frac{(\left| i-n_k \right|+1)^{\ell}}{2^k}, \quad \text{for all }\ell,k \in \N. $$
		Thus, $w_{i-n_k}\cdots w_{i-1}e_{i-n_k}\to 0$, when $k \to 0$. Moreover, there exists a constant $C'_{i}>0$ such that for all $\ell \in \N$ we have 
		$$\left\| \frac{e_{i + n_k}}{w_{i} \cdots w_{i + n_k - 1}} \right\|_\ell  \le C'_{i}\frac{(\left| i+n_k \right|+1)^{\ell}}{2^{n_k}}\to 0, \quad \text{when }k \to \infty.$$
		Then, $\frac{e_{i+n_k}}{w_{i}\cdots w_{i+n_k-1}}\to 0$ when $k\to \infty$. Therefore, by (\ref{equa3}), $B_w$ is  hypercyclic.
	\end{example}
	The literature already contains examples of distributionally chaotic shifts that are not hypercyclic; see, for instance, \cite[Example 13]{MOP}. In what follows, in order to illustrate the use of the sufficiency direction in Theorem \ref{T1} (in particular Corollary \ref{cor}), we present an example of a weighted backward shift on $\lambda_p(A, \Z)$ ($p \in \{0\} \cup [1, \infty)$), which is distributionally chaotic but not hypercyclic. 
	\begin{example}
		Consider the K\"othe matrix $A=(a_{j,k})_{j \in \Z,k \in \N}$ defined by
		$$\left( a_{j,k} \right)_{j \in \Z, k \in \N}:=(\underset{j \le0}{\underbrace{\cdots,1,1,1}},1,B_{1,k},1,1,2^k,B_{2,k},2^k,1, \cdots, 1,2^k,\cdots,n^k,B_{n,k},n^k,\cdots,2^k,1,\cdots),$$
		where 
		$$B_{n,k}:=(\underset{10^n\text{ times}}{\underbrace{(n+1)^k,(n+1)^k, \cdots, (n+1)^k}} ), \quad \text{for all }n,k\in \N.$$
		We denote by $I_n$ the set of indices that compose the block $B_{n,k}$ $(n,k\in\mathbb{N})$, note that these indices do not depend on $k$. For example $I_1=\left\{2,3,\cdots , 11  \right\}$. Now, consider the weighted backward shift $B_w$ on $\lambda_p(A, \Z)$ ($p\in \left\{ 0 \right\}\cup [1, \infty)$) with weights $w_{n}:=\frac{1}{2}$, for $n <0$ and $w_n:= 1$, for $n\ge 0$. 
		Since 
		$$\left\| \frac{e_{n}}{w_{0} \cdots w_{n - 1}} \right\|_k \ge 1, \quad \text{for all }n,k \in \N,$$
		then $B_w$ is not  hypercyclic. Now, we will prove that $B_w$ is distributionally chaotic. Note that for each $i \in \Z$ and $k \in \N$ there exists a constant $C>0$ such that
		$$\lim_{n \to \infty} a_{i-n,k}w_{i-n}\cdots w_{i-1} \le C. \lim_{n \to \infty} \frac{1}{2^n}=0.$$
		Then, conditions (A1) and (B1) of Corollary \ref{cor} are satisfied. Note that if $j= 2\sum_{\ell=1}^{N}\ell +\sum_{\ell=1}^{N}10^\ell$ for some $N \in \N$, then $a_{j,k}=1$ for all $k \in \N$. To prove conditions (A2) and (B2) we take $m=1$ and for each $k \in \N$ take $N_k \in \N$ where $N_k:=2\sum_{\ell=1}^{n_k}\ell +\sum_{\ell=1}^{n_k}10^\ell$, for some $n_k \in \N$ sufficiently large such that 
		\begin{equation}\label{eq}
			\frac{\sum_{\ell=k}^{n_k}10^\ell}{N_k}> (1-k^{-1}).
		\end{equation}
		Now, observe that 
		\begin{equation}\label{eqq}
			\frac{\left| a_{N_k-j,1}w_{N_k-j}\cdots w_{N_k-1} \right|}{\left| a_{N_k,k} \right|}= \left| a_{N_k-j,1} \right|> k,\quad \text{for all }j \in \N\text{ s.t. }N_k-j\in \bigcup_{i=k}^{n_k}I_{i}.
		\end{equation}
		Thus, by (\ref{eq}) and (\ref{eqq}) we obtain
		$$\text{card}\left\{ 1\le j\le N_k: \frac{\left| a_{N_k-j,1}w_{N_k-j}\cdots w_{N_k-1} \right|}{\left| a_{N_k,k} \right|}> k \right\}\ge \sum_{\ell=k}^{n_k}10^\ell> (1-k^{-1})N_k.$$
		Then, conditions (A2) and (B2) of Corollary \ref{cor} are satisfied with $r=1$, $i_{1,k}= N_k$ and $b_{1,k}=1$. Therefore, $B_w$ is distributionally chaotic.
	\end{example}
	
	\section{Mean Li--Yorke chaos}
	In the last decade, the study of average properties, such as mean equicontinuity
	and mean sensitivity, has become increasingly popular. In this context, the
	notion of mean Li--Yorke chaos has gained prominence; see, for instance,
	\cite{BBP,Nilson1,24,29,Chi} for recent works addressing this property.
	Below, we provide the formal definition of this concept.
	\begin{definition}\label{ldef}
		Let $(X,d)$ be a metric space and let $f:X\to X$ be a continuous map.
		A pair $(x,y)\in X\times X$, $x\neq y$, is called a \emph{mean Li--Yorke pair}
		for $f$ if
		\[
		\liminf_{n\to\infty} \frac1n \sum_{k=1}^{n} d\big(f^k(x), f^k(y)\big) = 0
		\quad\text{and}\quad
		\limsup_{n\to\infty} \frac1n \sum_{k=1}^{n} d\big(f^k(x), f^k(y)\big) > 0.
		\]
		The function $f$ is said to be \emph{mean Li--Yorke chaotic}
		if there exists an uncountable set $S\subset X$ such that every
		pair $(x,y)\in S\times S$ of distinct points is a mean Li--Yorke pair.
	\end{definition}
	We next present some definitions related to the notion of mean Li--Yorke chaos.
	\begin{definition}
		Let $(X,d)$ be a metric space and let $f:X\to X$ be a continuous map.
		\begin{itemize}
			\item[\rm (a)] We say that $f$ is \emph{mean sensitive} if there exists $\delta>0$ such that, for every $x\in X$ and every $\varepsilon>0$, there exists $y\in X$ with $d(x,y)<\varepsilon$ and
			\[
			\limsup_{n\to\infty}\frac{1}{n}\sum_{i=1}^{n} d(f^i( x),f^i( y))\ge \delta.
			\] 
			\item[\rm (b)] For a given positive number $\delta$, a pair $(x,y)\in X\times X$ is called a \emph{mean Li--Yorke $\delta$-chaotic pair} if
			\[
			\liminf_{n\to\infty}\frac{1}{n}\sum_{i=1}^{n} d(f^i (x),f^i( y))=0
			\quad \text{and} \quad
			\limsup_{n\to\infty}\frac{1}{n}\sum_{i=1}^{n} d(f^i( x),f^i (y))\ge \delta.
			\]
			We say that $f$ is \emph{mean Li--Yorke sensitive} if there exists $\delta>0$ such that, for every $x\in X$ and every $\varepsilon>0$, there exists $y\in X$ with $d(x,y)<\varepsilon$ such that $(x,y)$ is a mean Li--Yorke $\delta$-chaotic pair. 
		\end{itemize}
	\end{definition}
	It is clear that every mean Li--Yorke sensitive system is mean sensitive. From now on, given a Fréchet space $X$ endowed with a family of seminorms 
	$\bigl(\|\cdot\|_n\bigr)_{n\in\mathbb{N}}$, we will always consider the compatible 
	metric $d$ to be the one given in \eqref{d}. Recall that the compatible metric $d$ satisfies the following proprieties: 
	\begin{itemize}
		\item[(P1)] $d(x_1+x_2,0)\le d(x_1,0)+d(x_2,0)$ for all $x_1,x_2\in X$.
		\item[(P2)] $d(\lambda x_1,0)\le (1+|\lambda|)d(x_1,0)$ for all $\lambda\in\mathbb{K}$ and $x_1\in X$.
	\end{itemize}
	To prove our results in this section, we need the following definitions and lemmas from \cite{Chi}.
	\begin{definition} Let $X$ be a Fréchet space with the compatible metric $d$.
		Let $T : X \to X$ be a continuous linear operator. A vector $x\in X$ is called \textit{absolutely mean} \textit{semi-irregular} (or {\it semi-irregular point}) if
		$$\liminf_{n\to \infty}\frac{1}{n}\sum_{k=1}^{n}d(T^kx,0)=0\quad \text{and}\quad \limsup_{n\to \infty}\frac{1}{n}\sum_{k=1}^{n}d(T^kx,0)>0.$$
	\end{definition}
	\begin{definition}
		Let $T:X \to X$ be a continuous linear operator on a Fréchet space $X$ with the compatible metric $d$. The \textit{mean asymptotic cell} and the \textit{mean proximal cell} of $0$ are defined by 
		$$\text{MAsym}(T,0):=\left\{ x \in X: \lim_{n\to \infty} \frac{1}{n}\sum_{k=1}^{n}d(T^kx,0)=0 \right\}\quad \text{and}$$
		$$\text{MProx}(T,0):=\left\{ x \in X: \liminf_{n\to \infty} \frac{1}{n}\sum_{k=1}^{n}d(T^kx,0)=0 \right\}, \quad \text{respectively}.$$
	\end{definition}
	\begin{remark}\label{remark1}
		As observed in \cite{Chi}, \(\text{MProx}(T,0)\) is a \(G_\delta\) set. Moreover, \cite[Lemma~4.7]{Chi} proves that if \(\text{MAsym}(T,0)\) is residual, then \(\text{MAsym}(T,0)\) coincides with the whole space.
	\end{remark}
	\begin{lemma}\cite[Proposition 4.11]{Chi}\label{Le0}
		Let $X$ be a Fr\'echet space, let $T:X\to X$ be a continuous linear operator, and let $d$ be a compatible metric on $X$. Then the following assertions are equivalent:
		\begin{itemize}
			\item[\rm (a)] $T$ is mean sensitive;
			
			\item[\rm (b)] there exist a sequence $(y_k)_k$ in $X$ and an increasing sequence $(N_k)_k$ in $\mathbb{N}$ such that $\lim_{k\to\infty} y_k=0$ and
			\[
			\inf_{k\in\mathbb{N}} \frac{1}{N_k}\sum_{i=1}^{N_k} d(T^i y_k,0)>0.
			\]
		\end{itemize}
	\end{lemma}
	
	\begin{lemma}\cite[Theorem 4.15]{Chi}\label{L1}
		Let $T:X \to X$ be a continuous linear operator on a Fréchet space $X$. Then the following statements are equivalent:
		\begin{itemize}
			\item[\rm (a)] $T$ is mean Li--Yorke chaotic; 
			\item[\rm (b)] $T$ is mean Li--Yorke chaotic sensitive;
			\item[\rm (c)] $T$ admits an absolutely mean semi-irregular vector.
		\end{itemize}
	\end{lemma}
	\begin{lemma}\cite[Theorem 4.27]{Chi}\label{L3}
		Let $T:X \to X$ be a continuous linear operator on a Fréchet space $X$ with the compatible metric $d$. Then,  the following statements are equivalent:
		\begin{enumerate}
			\item[\rm (a)] $T$ admits a dense set of absolutely mean semi-irregular vectors;
			\item[\rm (b)] the mean proximal cell of $0$ is dense in $X$ and there exists $x \in X$ such that
			\[
			\limsup_{n \to \infty} \frac{1}{n} \sum_{i=1}^{n} d(T^{i}x,0) > 0;
			\]
		\end{enumerate}
	\end{lemma}
	
	Let $X$ be a Fréchet sequence space over $\Z$ with basis $(e_n)_{n \in \Z}$. Consider the following condition:
	\begin{itemize}
		\item[\rm (C)] For each $n\in \N$, $m \in \Z$ and $x=(x_j)_{j \in \Z}$, we have:
		$$\left| x_m \right|\left\| e_m \right\|_n \le \left\| x \right\|_n.$$
	\end{itemize}
	\begin{lemma}\label{lemex}
		Let $X$ be a Fréchet sequence space over $\mathbb{Z}$, with the compatible metric $d$. Suppose that the sequence 
		$(e_n)_{n \in \mathbb{Z}}$ of canonical vectors is a basis. Suppose that the bilateral weighted backward shift $B_w$, with nonzero weights $w := (w_n)_{n \in \mathbb{Z}}$, is well-defined and continuous on $X$. Then the following statements are equivalent:
		\begin{itemize}
			\item[\rm (a)] $\liminf_{n\to \infty}\frac{1}{n}\sum_{k=1}^{n}d(w_{j-k}\cdots w_{j-1}e_{j-k},0)=0,$ for some $j \in \Z$; 
			\item[\rm (b)] $\liminf_{n\to \infty}\frac{1}{n}\sum_{k=1}^{n}d(w_{i-k}\cdots w_{i-1}e_{i-k},0)=0,$ for all $i \in \Z$. 
		\end{itemize}
	\end{lemma}
	\begin{proof}
		It is obvious that (b) implies (a). Now, suppose that (a) holds for some $j\in \Z$. Then, there exists an increasing sequence $(n_k)_{k\in \N}$ of positive integers such that 
		$$\lim_{k\to \infty}\frac{1}{n_k}\sum_{i=1}^{n_k}d(w_{j-i}\cdots w_{j-1}e_{j-i},0)=0.$$
		Take $\ell \in \Z$ with $\ell>j$ and define $N_k:= n_k+(\ell-j)$, for each $k \in \N$. Then
		\begin{align*}
			\frac{1}{N_k}\sum_{i=1}^{N_k}d(w_{\ell-i}\cdots w_{\ell-1}e_{\ell-i},0)&=\frac{1}{N_k}\sum_{i=1}^{\ell-j}d(w_{\ell-i}\cdots w_{\ell-1}e_{\ell-i},0)+ \frac{1}{N_k}\sum_{i=\ell-j +1}^{N_k}d(w_{\ell-i}\cdots w_{\ell-1}e_{\ell-i},0)\\
			&\le \frac{\ell-j}{N_k}+  \frac{1}{N_k}\sum_{i=\ell-j +1}^{N_k}d(w_{\ell-i}\cdots w_{j-1}\overset{:=C}{\overbrace{w_{j}\cdots w_{\ell-1}}}e_{\ell-i},0)\\ 
			&=\frac{\ell-j}{N_k} + \frac{1}{N_k}\sum_{r=1}^{n_k}d(C\,w_{j-r}\cdots w_{j-1}e_{j-r},0)\\ 
			&\le \frac{\ell-j}{N_k} + \frac{n_k (1+|C|)}{N_k}\frac{1}{n_k}\sum_{r=1}^{n_k}d(w_{j-r}\cdots w_{j-1}e_{j-r},0)\overset{k\to \infty}{\longrightarrow} 0.
		\end{align*}
		Now, take $\ell \in \Z$ with $\ell<j$ and define $N_k:= n_k-(j-\ell)$, for each $k \in \N$. Suppose that $N_k>0$, for all $k\in \N$, otherwise, it suffices to discard the finitely many negative terms and reindex the sequence. Then
		\begin{align*}
			\frac{1}{N_k}\sum_{i=1}^{N_k}d(w_{\ell-i}\cdots w_{\ell-1}e_{\ell-i},0)&=\frac{1}{N_k}\sum_{i=1}^{N_k}d(\overset{:=C}{\overbrace{\frac{1}{w_{\ell}\cdots w_{j-1}}}}w_{\ell-i}\cdots w_{\ell-1}w_{\ell}\cdots w_{j-1} e_{\ell-i},0)\\
			&\le \frac{(1+|C|)}{N_k}\sum_{i=1}^{N_k}d(w_{\ell-i}\cdots w_{j-1} e_{\ell-i},0)\\
			&\le \frac{(1+|C|)n_k}{N_k}\frac{1}{n_k}\sum_{r=(j-\ell)+1}^{n_k}d(w_{j-r}\cdots w_{j-1} e_{j-r},0)\overset{k\to \infty}{\longrightarrow} 0.
		\end{align*}
	\end{proof}
	
	\begin{theorem}\label{T2}
		Let $X$ be a Fréchet sequence space over $\mathbb{Z}$, endowed with the compatible metric $d$. Suppose that the sequence 
		$(e_n)_{n \in \mathbb{Z}}$ of canonical vectors is a basis and that condition (C) holds. Suppose that the bilateral weighted backward shift $B_w$, with nonzero weights $w := (w_n)_{n \in \mathbb{Z}}$, is well-defined and continuous on $X$. Then, $B_w$ is mean Li--Yorke chaotic if and only if the following conditions hold:
		\begin{itemize}
			\item[\rm (A)] $\liminf_{n\to \infty}\frac{1}{n}\sum_{k=1}^{n}d(w_{-k}\cdots w_{-1}e_{-k},0)=0;$
			\item[\rm (B)]  there exist $\varepsilon>0$ and an increasing sequence $(N_k)_{k \in \N}$ of positive integers, such that for each $k\in \N$ there are $r:= r(k) \in \N$ and scalars $b_{-r,k},\cdots , b_{r,k} \in \K$ with $d \left( \sum_{j=-r}^{r} b_{j,k}e_{j},0 \right)<\frac{1}{k}$ and
			$$\frac{1}{N_k}\sum_{i=1}^{N_k}d \left( \sum_{j=-r}^{r} b_{j,k}w_{j-i}\cdots w_{j-1} e_{j-i},0 \right)\ge \varepsilon.$$
		\end{itemize}
	\end{theorem}
	\begin{proof}
		($\Rightarrow$) Suppose that $B_w$ is mean Li--Yorke chaotic. Therefore, by Lemma \ref{L1}, $B_w$ admits an absolutely mean semi-irregular vector $y:=(y_j)_{j \in \Z}$, then
		\begin{equation}\label{eq4}
			\liminf_{n\to \infty}\frac{1}{n}\sum_{k=1}^{n}d((B_w)^k(y),0)=0\quad \text{and}\quad \limsup_{n\to \infty}\frac{1}{n}\sum_{k=1}^{n}d((B_w)^k(y),0)>0.
		\end{equation}
		Take $j_0 \in \Z$ such that $y_{j_0}\ne 0$. Then, by condition (C) and (\ref{eq4}), we have 
		$$\liminf_{n\to \infty}\frac{1}{n}\sum_{k=1}^{n}d(y_{j_0}w_{j_0-k}\cdots w_{j_0-1}e_{j_0-k},0)\le\liminf_{n\to \infty}\frac{1}{n}\sum_{k=1}^{n}d((B_w)^k(y),0)= 0.$$
		Therefore, by Lemma \ref{lemex}, condition (A) holds. Since $B_w$ is mean Li--Yorke chaotic, then, by Lemma \ref{L1}, $B_w$ is mean Li-Yorke sensitive and, in particular, it is mean sensitive. Thus, by Lemma \ref{Le0}, there exist $\varepsilon\in (0,1)$, a sequence $(z_k)_k$ in $X$ and an increasing sequence $(N_k)_k$ in $\mathbb{N}$ such that $\lim_{k\to\infty} z_k=0$ and
		$$\frac{1}{N_k}\sum_{i=1}^{N_k} d((B_w )^i(z_k),0)>\varepsilon, \quad \text{for all }k\in\N.$$
		Take an increasing sequence $(k_s)_{s \in \N}$ of positive integers such that $d(z_{k_s},0)< \frac{1}{2s}$. Define $x_s:=z_{k_s}$ and $M_s:=N_{k_s}$ for each $s \in \N$. Then
		\begin{equation}\label{bib1}
			\frac{1}{M_s}\sum_{i=1}^{M_s} d((B_w)^{i}(x_s),0)>\varepsilon, \quad \text{for all }s\in\N.
		\end{equation}
		Fix $s \in \N$ and write $x_s=(x_{s,n})_{n\in \Z}$. Since $\lim_{n\to \infty } \sum_{j=-n}^{n}x_{s,j}e_j=x_s$, then, by the continuity of $B_w$, there exists $r\in \N$ such that 
		\begin{equation}\label{bib2}
			d\left( (B_w)^i\left( \sum_{j=-r}^{r}x_{s,j}e_j \right),(B_w)^i(x_s) \right)< \frac{\varepsilon}{2s}, \quad \text{for all }i \in \left\{ 0,1,\cdots , M_s \right\}.
		\end{equation}
		Therefore, by (\ref{bib1}) and (\ref{bib2}), and using the triangle inequality of $d$ we obtain 
		$$d\left( \sum_{j=-r}^{r}x_{s,j}e_j,0 \right)<\frac{1}{s}\quad \text{and}\quad \frac{1}{M_s}\sum_{i=1}^{M_s} d\left( \sum_{j=-r}^{r} b_{j,k}w_{j-i}\cdots w_{j-1} e_{j-i},0 \right)>\frac{\varepsilon}{2}.$$
		($\Leftarrow$) By condition (A) and Lemma \ref{lemex}, we have that 
		$$\liminf_{n\to \infty}\frac{1}{n}\sum_{k=1}^{n}d((B_w)^k(e_j),0)=0, \quad \text{for all }j \in \Z.$$
		If there exists $j_0 \in \Z$ such that $\limsup_{n\to \infty}\frac{1}{n}\sum_{k=1}^{n}d((B_w)^k(e_{j_0}),0)>0$, then $B_w$ is mean Li--Yorke chaotic, given that $e_{j_0}$ is an absolutely mean semi-irregular vector. Otherwise, we have
		$$\lim_{n\to \infty}\frac{1}{n}\sum_{k=1}^{n}d((B_w)^k(e_j),0)=0, \quad \text{for all }j \in \Z.$$
		Therefore, using properties (P1) and (P2) of the metric $d$, we get $c_{00}(\Z)\subset \text{MProx}(B_w,0)$. Therefore, $\text{MProx}(B_w,0)$ is dense. By condition (B) and Lemma \ref{Le0}, we have that $B_w$ is mean sensitive. Then, there exists $x\in X$ such that
		$$\limsup_{n\to \infty}\frac{1}{n}\sum_{k=1}^{n}d((B_w)^k(x),0)>0.$$
		Thus, by Lemma \ref{L3}, $B_w$ admits a dense set of absolutely mean semi-irregular vectors. Therefore, by Lemma \ref{L1}, we conclude that $B_w$ is mean Li--Yorke chaotic.
	\end{proof}
	\begin{remark}
		The condition (A) in the last theorem was used in the part $(\Leftarrow)$ of the proof to ensure that $\text{MProx}(B_w,0)$ is dense. Since in the unilateral case
		$$\lim_{n \to \infty} (B_w)^n(e_i)=0, \quad \text{for all } i \in \N,$$
		then it follows immediately that $\text{MProx}(B_w,0)$ is dense. Therefore, in this setting, we obtain the following characterization: 
	\end{remark}
	\begin{theorem}\label{theo3}
		Let $X$ be a Fréchet sequence space, endowed with the compatible metric $d$. Suppose that the sequence 
		$(e_n)_{n \in \mathbb{N}}$ of canonical vectors is a basis. Suppose that the unilateral weighted backward shift $B_w$, with nonzero weights $w := (w_n)_{n \in \mathbb{N}}$, is well-defined and continuous on $X$. Then $B_w$ is mean Li--Yorke chaotic if and only if the following condition holds:
		\begin{itemize}
			\item There exist $\varepsilon>0$ and an increasing sequence $(N_k)_{k \in \N}$ of positive integers, such that for each $k\in \N$ there are $r:= r(k) \in \N$ and scalars $b_{1,k},\cdots , b_{r,k} \in \K$ with $d \left( \sum_{j=1}^{r} b_{j,k}e_{j},0 \right)<\frac{1}{k}$ and
			$$\frac{1}{N_k}\sum_{i=1}^{N_k}d \left( \sum_{j=1}^{r} b_{j,k}w_{j-i}\cdots w_{j-1} e_{j-i},0 \right)\ge \varepsilon,$$
			where we consider $e_{k}= (0)_{j \in \N}$ and $w_k=0$, for $k<1$.
		\end{itemize}
	\end{theorem}
	As in the case of Li–Yorke chaos (see \cite[Corollary 20]{Nilson2}), we present below propositions that reveal a dichotomy of weighted shifts with respect to mean Li–Yorke chaos.
	\begin{proposition}
		Let $X$ be a Fréchet sequence space over $\mathbb{Z}$, endowed with the compatible metric $d$. Suppose that the sequence 
		$(e_n)_{n \in \mathbb{Z}}$ of canonical vectors is a basis. Suppose that the bilateral weighted backward shift $B_w$, with nonzero weights $w := (w_n)_{n \in \mathbb{Z}}$, is well-defined and continuous on $X$. If $\liminf_{n\to \infty}\frac{1}{n}\sum_{k=1}^{n}d(w_{-k}\cdots w_{-1}e_{-k},0)=0$, then either 
		\begin{itemize}
			\item[\rm (a)] $B_w$ is mean Li--Yorke chaotic, or
			\item[\rm (b)] $\lim_{n\to \infty}\frac{1}{n}\sum_{k=1}^{n}d((B_w)^k(x),0)=0$, for all $x\in X$.
		\end{itemize}
	\end{proposition}
	\begin{proof}
		Suppose that $B_w$ is not mean Li--Yorke chaotic. Then $\text{MProx}(T,0)=\text{MAsym}(T,0)$. Since $$\lim_{n\to \infty}\frac{1}{n}\sum_{k=1}^{n}d((B_w)^k(e_0),0)=\liminf_{n\to \infty}\frac{1}{n}\sum_{k=1}^{n}d(w_{-k}\cdots w_{-1}e_{-k},0)=0,$$ by Lemma \ref{lemex} and proprieties (P1) and (P2) of the compatible metric $d$, we get $c_{00}(\Z)\subset \text{MProx}(T,0)$. Therefore, by Remark \ref{remark1}, we have that $\text{MProx}(T,0)=\text{MAsym}(T,0)$ is residual. Thus, again by Remark \ref{remark1}, $\text{MAsym}(T,0)=X$.
	\end{proof}
	\begin{proposition}
		Let $X$ be a Fréchet sequence space, endowed with the compatible metric $d$. Suppose that the sequence 
		$(e_n)_{n \in \mathbb{N}}$ of canonical vectors is a basis. Suppose that the unilateral weighted backward shift $B_w$, with nonzero weights $w := (w_n)_{n \in \mathbb{N}}$, is well-defined and continuous on $X$. Then either 
		\begin{itemize}
			\item[\rm (a)] $B_w$ admits a dense set of absolutely mean semi-irregular vectors, or
			\item[\rm (b)] $\lim_{n\to \infty}\frac{1}{n}\sum_{k=1}^{n}d((B_w)^k(x),0)=0$, for all $x\in X$.
		\end{itemize}
	\end{proposition}
	\begin{proof}
		Suppose that item (b) is false. Then, there exists $x \in X$ such that $$\limsup_{n\to \infty}\frac{1}{n}\sum_{k=1}^{n}d((B_w)^k(x),0)>0.$$ Since $\lim_{n\to \infty}\frac{1}{n}\sum_{k=1}^{n}d((B_w)^k(y),0)=0$, for all $y \in c_{00}(\N)$, then, by Lemma \ref{L3}, item (a) holds. 
	\end{proof}
	By using Theorem \ref{T2}, we obtain characterizations of mean Li-Yorke chaos in the context of weighted backward shifts on the Köthe sequence spaces $\lambda_p(A,\mathbb{Z})$ with $p \in [1,\infty) \cup \{0\}$, since these spaces satisfy condition (C). Furthermore, by using Theorem \ref{theo3}, we obtain the corresponding unilateral characterizations in this setting. Next, we will obtain results involving the increasing sequence of seminorms $(\|\cdot\|_n)_{n \in \mathbb{N}}$ that induce the topology of $X$. To do this, we need the following definition and lemma: 
	\begin{definition}
		Let $X$ be a Fréchet space endowed with an increasing sequence
		$(\|\cdot\|_k)_{k\in\mathbb{N}}$ of seminorms and with the compatible metric $d$. A vector $x\in X$ is called \textit{absolutely mean} $m$-\textit{irregular} if 
		$$\liminf_{n\to \infty}\frac{1}{n}\sum_{k=1}^{n}d(T^kx,0)=0\quad \text{and}\quad \limsup_{n\to \infty}\frac{1}{n}\sum_{k=1}^{n}\left\| T^k x \right\|_m=\infty.$$ 
	\end{definition}
	\begin{lemma}\cite[Theorem 4.29]{Chi}\label{L0}
		Let $T:X \to X$ be a continuous linear operator on a Fréchet space $X$. Then the set of absolutely 
		mean semi-irregular vectors is contained in the closure of the set of 
		absolutely mean $m$-irregular vectors for some $m \in \mathbb{N}$.
	\end{lemma}
	\begin{proposition}\label{prop}
		Let $X$ be a Fréchet sequence space over $\mathbb{Z}$, endowed with an increasing sequence $(\left\| \cdot \right\|_n)_{n \in \N}$ of seminorms and with the compatible metric $d$. Suppose that the sequence 
		$(e_n)_{n \in \mathbb{Z}}$ of canonical vectors is a basis and that condition (C) holds. Suppose that the bilateral weighted backward shift $B_w$, with nonzero weights $w := (w_n)_{n \in \mathbb{Z}}$, is well-defined and continuous on $X$. If $B_w$ is mean Li--Yorke chaotic, then the following conditions hold:
		\begin{itemize}
			\item[\rm (A)] $\liminf_{n\to \infty}\frac{1}{n}\sum_{k=1}^{n}d(w_{-k}\cdots w_{-1}e_{-k},0)=0;$
			\item[\rm (B)]  there exist $m \in \N$ and an increasing sequence $(N_k)_{k \in \N}$ of positive integers, such that for each $k\in \N$ there are $r:= r(k) \in \N$ and scalars $b_{-r,k},\cdots , b_{r,k} \in \K$ with $\left\|  \sum_{j=-r}^{r} b_{j,k}e_{j} \right\|_{p(k)} > 0$ and
			$$\frac{1}{N_k\left\|  \sum_{j=-r}^{r} b_{j,k}e_{j} \right\|_{p(k)}}\sum_{i=1}^{N_k}\left\| \sum_{j=-r}^{r} b_{j,k}w_{j-i}\cdots w_{j-1}e_{j-i} \right\|_m\ge k,$$
			where $p(k):= m$, if $1\le k\le m$, and $p(k):= k$, if $k>m$.
		\end{itemize}        
	\end{proposition}
	\begin{proof}
		The proof of item (A) is the same as the one given for item (A) in the proof of Theorem ~\ref{T2}, so we only prove item (B). Since $B_w$ is mean Li-Yorke chaotic, then, by Lemma \ref{L1}, $B_w$ admits an absolutely mean semi-irregular vector. Therefore, Lemma \ref{L0} ensures the existence of $m \in \N$ and $x=(x_j)_{j \in \Z} \in X$ such that
		\begin{equation}\label{on}
			\limsup_{n\to \infty}\frac{1}{n}\sum_{k=1}^{n}\left\| (B_w)^k (x) \right\|_m=\infty.
		\end{equation}
		It is easy to see that we can take $m\in \N$ in (\ref{on}) sufficiently large such that $\left\| x \right\|_m \neq 0$. By (\ref{on}), there exists an increasing sequence $(N_k)_{k \in \N}$ of positive integers such that for each $k \in \N$
		\begin{equation}\label{eq6.}
			\frac{1}{N_k\left\| x \right\|_{p(k)}}\sum_{i=1}^{N_k}\left\| (B_w)^i (x) \right\|_m> k +\frac{1}{2}.
		\end{equation}
		Since $\sum_{s=-j}^{j} x_s e_s \to x$ as $j\to\infty$ and $B_w$ is continuous, there exists $r \in \N$ such that 
		\begin{equation}\label{eq5}
			\left| \frac{\left\| (B_w)^i\left( \sum_{s=-r}^{r}x_se_s \right) \right\|_m}{\left\| \sum_{s=-r}^{r}x_se_s \right\|_{p(k)}}-\frac{\left\| (B_w)^i(x) \right\|_m}{\left\| x \right\|_{p(k)}} \right|< \frac{1}{2}, \quad \text{for all } i=1, \cdots , N_k.
		\end{equation}
		Thus, by (\ref{eq6.}) and (\ref{eq5}), we conclude condition (B).
	\end{proof}
	We say that $X$ is a Banach sequence space over $\Z$ if $(X,\|\cdot\|)$ is a Banach space which is a vector subspace of $\mathbb{K}^{\mathbb{Z}}$ and the inclusion map $X \to \mathbb{K}^{\mathbb{Z}}$ is continuous, i.e., convergence in $X$ implies coordinatewise convergence. Under the assumption that $X$ is a Banach sequence space over $\mathbb{Z}$, we can improve Proposition~\ref{prop}. To do so, we will need the following lemma: 
	\begin{lemma}\cite[Theorems 5 and 9]{BBP}\label{LL}
		Let $(X,\|\cdot\|)$ be a Banach space and $T:X\to X$ a continuous linear operator. Then, the following assertions are equivalent:
		\begin{itemize}
			\item[\rm (a)] $T$ is mean Li--Yorke chaotic;
			\item[\rm (b)] $T$ admits an \emph{absolutely mean-irregular vector} $x \in X$, that is, 
			$$\liminf_{N\to\infty}\frac{1}{N}\sum_{j=1}^{N}\|T^{j}x\|=0
			\quad \text{and} \quad
			\limsup_{N\to\infty}\frac{1}{N}\sum_{j=1}^{N}\|T^{j}x\|=\infty;$$
			\item[\rm (c)] $T$ satisfies the \emph{Mean Li--Yorke Chaos Criterion (MLYCC)}, that is, there exists a subset $X_0$ of $X$ with the following properties:
			\begin{enumerate}
				\item[\rm (i)] $\displaystyle \liminf_{N\to\infty}\frac{1}{N}\sum_{j=1}^{N}\|T^{j}x\|=0$, for every $x\in X_0$;
				
				\item[\rm (ii)] there are sequences $(y_k)$ in $\overline{\operatorname{span}(X_0)}$ and $(N_k)$ in $\mathbb{N}$ such that
				\[
				\frac{1}{N_k}\sum_{j=1}^{N_k}\|T^{j}y_k\|>k\|y_k\|,
				\quad \text{for every } k\in\mathbb{N}.
				\]
			\end{enumerate}
		\end{itemize}
	\end{lemma}
	\begin{theorem}\label{ttt1}
		Let $(X,\|\cdot\|)$ be a Banach sequence space over $\Z$. Suppose that the sequence 
		$(e_n)_{n \in \mathbb{Z}}$ of canonical vectors is a basis and that condition (C) holds with the norm $\|\cdot\|$. Suppose that the bilateral weighted backward shift $B_w$, with nonzero weights $w := (w_n)_{n \in \mathbb{Z}}$, is well-defined and continuous on $X$. Then, $B_w$ is mean Li--Yorke chaotic if and only if the following conditions hold: 
		\begin{itemize}
			\item[\rm (A)] $\displaystyle \liminf_{n\to \infty}\frac{1}{n}\sum_{k=1}^{n}\|w_{-k}\cdots w_{-1}e_{-k}\|=0;$
			
			\item[\rm (B)] there exist $m \in \N$ and an increasing sequence $(N_k)_{k \in \N}$ of positive integers such that, for each $k\in \N$, there are $r:= r(k) \in \N$ and scalars $b_{-r,k},\cdots , b_{r,k} \in \K$ with $\left\| \sum_{j=-r}^{r} b_{j,k}e_{j} \right\| > 0$ and
			$$
			\frac{1}{N_k\left\| \sum_{j=-r}^{r} b_{j,k}e_{j} \right\|}\sum_{i=1}^{N_k}\left\| \sum_{j=-r}^{r} b_{j,k}w_{j-i}\cdots w_{j-1}e_{j-i} \right\|\ge k.
			$$
		\end{itemize}
	\end{theorem}
	\begin{proof}
		$(\Rightarrow)$ By the Lemma \ref{LL}, $B_w$ admits an absolutely mean irregular vector $x:=(x_n)_{n\in \Z}\in X$. In particular, $\liminf_{n\to \infty}\frac{1}{n}\sum_{k=1}^{n}\left\| (B_w)^k(x) \right\|=0.$ Take $j_0\in\Z$ such that $x_{j_0}\ne 0$. Then, by condition (C), we have 
		$$\liminf_{n\to \infty}\frac{1}{n}\sum_{k=1}^{n}\|x_{j_0}w_{j_0-k}\cdots w_{j_0-1}e_{j_0-k}\|\le\liminf_{n\to \infty}\frac{1}{n}\sum_{k=1}^{n}\left\| (B_w)^k(x) \right\|=0.$$
		By an analogous argument to that of Lemma ~\ref{lemex} (with the norm replacing the metric), we obtain item~(A). Since $\limsup_{N\to\infty}\frac{1}{N}\sum_{j=1}^{N}\|T^{j}x\|=\infty$, then there exists an increasing sequence $(N_k)_{k \in \N}$ of positive integers such that for each $k \in \N$
		$$\frac{1}{N_k\left\| x \right\|}\sum_{i=1}^{N_k}\left\| (B_w)^i (x) \right\|> k +\frac{1}{2}.$$
		Therefore, using analogous arguments to those used in the $(\Rightarrow)$ part of the proof of Theorem~\ref{T2}—namely, the density of \(c_{00}(\mathbb{Z})\) (since \((e_n)_{n\in\mathbb{Z}}\) is a basis) and the continuity of \(B_w\)—we conclude item~(B). 
		
		$(\Leftarrow)$ By item (A) and by an argument analogous to that of Lemma ~\ref{lemex} (with the norm replacing the metric), we have $\liminf_{n\to \infty}\frac{1}{n}\sum_{k=1}^{n}\left\| (B_w)^k(e_i) \right\|=0$, for all $i \in \Z$. Then, taking $X_0:=\left\{ e_n : n\in \Z\right\}$, we obtain item (i) of MLYCC. Moreover, taking $y_k:= \sum_{j=-r}^{r} b_{j,k}e_{j}$ for each $k\in \N$, by item (B), we conclude  item (ii) of MLYCC. Therefore, by Lemma \ref{LL}, $B_w$ is mean Li--Yorke chaotic.
	\end{proof}
	Examples of shifts that are hypercyclic but not mean Li--Yorke chaotic are already known; see, for instance, \cite[Example 23]{BBPW}. In what follows, to illustrate the use of the necessary direction in Theorem \ref{T2}, we will give an example of a weighted backward shift on $s(\Z)$ that is  hypercyclic, but it is not mean Li--Yorke chaotic.
	\begin{example}
		Consider the weighted backward shift $B_w$ on $s(\Z)$ with weights 
		$$(w_n)_{n \in \Z}:= (\cdots,\underset{\text{Block }C_n}{\underbrace{\overset{n\text{ times}}{\overbrace{2,\cdots,2}},\overset{n \text{ times}}{\overbrace{\frac{1}{2},\cdots,\frac{1}{2}}}}},\overset{n^2\text{ times}}{\overbrace{\underset{\text{Block }B_n}{\underbrace{1,\cdots,1}}}},\cdots,\underset{\text{Block }C_1}{\underbrace{2,\frac{1}{2}}},\underset{\text{Block }B_1}{ \underbrace{1}},\underset{j\ge 0}{\underbrace{2,2,2,2, \cdots}}),$$
		Consider the sequence $(n_k)_{k \in \N}$ of positive integers, given by 
		$$n_k:= k+\sum_{j=0}^{k-1}2j+\sum_{j=1}^{k}j^2, \quad k \in \N.$$
		Fix $j \in \Z$. Then, there exists a constant $C_{j}>0$ such that for all $\ell \in \N$
		$$\left\| w_{j-n_k}\cdots w_{j-1}e_{j-n_k}  \right\|_{\ell} \le C_{j}\frac{\left( \left| j-n_k \right|+1 \right)^{\ell}}{2^k}, \quad \text{for all }k \in \N.$$
		Thus, $ w_{\ell-n_k}\cdots w_{\ell-1} e_{\ell - n_k}\to 0$, when $k \to \infty$. On the other hand, there is a constant $C'_{j}>0$ such that for all $\ell \in \N$ 
		$$\left\| \frac{e_{j+n_k}}{ w_{j}\cdots w_{j+n_k -1} } \right\|_{\ell}\le C'_{j}\frac{\left( \left| j+n_k \right|+1 \right)^{\ell}}{2^{n_k}}\quad \text{for all } k \in \N.$$
		Thus, $\frac{e_{\ell+n_k}}{  w_{\ell}\cdots w_{\ell+n_k -1} } \to 0$, when $k\to \infty$. Therefore, $B_w$ is hypercyclic. \\
		Now, we will prove that $B_w$ is not mean Li-Yorke chaotic. Fix $N>3$. Then, there exists $n\in \N$ such that 
		$$\sum_{k=1}^{n+1} k^2 +\sum_{k=1}^{n+1}2k\ge N\ge \sum_{k=1}^{n} k^2 +\sum_{k=1}^{n}2k.$$
		Then, 
		\begin{align*}
			\frac{1}{N}\sum_{k=1}^{N}d(w_{-k}\cdots w_{-1}e_{-k},0)&=\frac{1}{N}\sum_{k=1}^{N}\sum_{j=1}^{\infty}\frac{1}{2^j}\min\left\{ 1,w_{-k}\cdots w_{-1}(k+1)^j \right\} \\ 
			&\ge\frac{\sum_{k=1}^{n} k^2}{\sum_{k=1}^{n+1} k^2 +\sum_{k=1}^{n+1}2k} \ge \lambda,
		\end{align*}
		where $\lambda$ is a suitable positive constant. Thus, the condition (A) of Theorem \ref{T2} is not satisfied. Therefore, $B_w$ is not mean Li--Yorke chaotic.
	\end{example}
	In the following, to illustrate the use of the sufficiency direction in Theorem \ref{ttt1}, we will give an example of a weighted backward shift on $\ell^p(\nu,\Z)$ ($p\in [1, \infty)$) that is mean Li-Yorke chaotic but is not  hypercyclic. We emphasize that examples of operators that are mean Li–Yorke chaotic but not hypercyclic are already known; see, for instance, a remark on page 11 of \cite{BBP}.
	\begin{example}
		Consider the weighted backward shift $B_w$ on $\ell^p(\nu,\Z)$ ($p\in [1, \infty)$) (or $c_0(\nu,\Z)$) with 
		$$\nu=(v_n)_{n\in \Z}:=(\underset{j \le0}{\underbrace{\cdots,1,1,1,1}},1,B_{1},1,1,2,B_{2},2,1, \cdots, 1,2,\cdots,n,B_{n},n,\cdots,2,1,\cdots),$$
		where 
		$$B_{n}:=(\underset{10^n\text{ times}}{\underbrace{n+1,n+1 \cdots, n+1}} ), \quad \text{for all }n\in \N$$
		and weights $(w_n)_{n\in \Z}$ given by $w_n:=\frac{1}{2}$ if $n <0$ and $w_n:= 1$ if $n\ge 0$. It is easy to see that $B_w$ is not  hypercyclic, given that $(e_{n_k})_{k \in \N}$ does not converge to zero for any increasing sequence $(n_k)_{k \in \N}$ of positive integers.  \\ 
		Now, we will prove that $B_w$ is mean Li--Yorke chaotic. Since 
		\begin{equation}\label{01}
			\lim_{N\to \infty}\frac{1}{N}\sum_{k=1}^{N}\left\| w_{-k}\cdots w_{-1}e_{-k} \right\|=\lim_{N\to \infty}\frac{1}{N}\sum_{k=1}^{N}\frac{1}{2^k}=0,
		\end{equation}
		then condition (A) of Theorem \ref{ttt1} is satisfied. Now, for each $k \in \N$ take $N_k:= 2\sum_{\ell=1}^{k}\ell +\sum_{\ell=1}^{k}10^\ell$. Note that $v_{N_k}=1$ for all $k\in \N$. Then 
		\begin{align*}
			\frac{1}{N_k\left\| e_{N_k} \right\|}\sum_{i=1}^{N_k}\left\| (B_w)^i(e_{N_k}) \right\|=\frac{1}{N_k\left| \nu_{N_k} \right|}\sum_{i=1}^{N_k}\left| \nu_{N_k-i} \right|&= \frac{\left( 2\sum_{j=1}^{k}\sum_{i=1}^{j}i \right)+\left( \sum_{j=1}^{k}\sum_{i=1}^{10^j}(j+1) \right)}{2\sum_{\ell=1}^{k}\ell +\sum_{\ell=1}^{k}10^\ell} \\
			& = \frac{ \frac{2k(k+1)(k+2)}{6} + \frac{(9k+8)10^{k+1}-80}{81} }{ k(k+1) + \frac{10^{k+1}-10}{9} } \\
			& > \frac{\frac{k\cdot10^{k+1}}{9}}{\frac{2\cdot10^{k+1}}{9}}=\frac{k}{2}
		\end{align*}
		Thus, condition (B) of Theorem \ref{ttt1} is satisfied. Therefore, $B_w$ is mean Li-Yorke chaotic.
		
	\end{example}
	\begin{remark}
		If the uncountable set \(S\) in Definition \(\ref{ldef}\) can be chosen to be dense, then we say that \(f\) is {\it densely mean Li--Yorke chaotic}. Moreover, we say that a continuous linear operator \(T\) on a Banach space \(X\) satisfies the {\it Dense Mean Li--Yorke Chaos Criterion} (DMLYCC) if it satisfies conditions (i) and (ii) in item (c) of Lemma \(\ref{LL}\), with \(X_0\) being a dense set. Suppose that $X$ is separable, then, by \cite[Theorem 21]{BBP}, we have that \(T\) is densely mean Li--Yorke chaotic if and only if \(T\) satisfies the DMLYCC.
		
		Since in the last example we proved that
		\[
		\lim_{N\to \infty}\frac{1}{N}\sum_{k=1}^{N}\left\| w_{-k}\cdots w_{-1}e_{-k} \right\|=0,
		\]
		by an analogous lemma to Lemma \(\ref{lemex}\) (with the norm in place of the metric, and the limit instead of the limit inferior), we may conclude that
		\[
		\lim_{N\to \infty}\frac{1}{N}\sum_{k=1}^{N}\left\| (B_w)^k(x) \right\|=0,
		\quad \text{for every } x \in c_{00}(\mathbb{Z}).
		\]
		Therefore, we may take \(X_0\) to be a dense set, and hence \(B_w\) is in fact {\it densely mean Li--Yorke chaotic}.
	\end{remark}
	Recall that an operator $T$ on a Banach space $Y$ is said to be
	\emph{absolutely Cesàro bounded} if there exists a constant $C \in (0,\infty)$ such that
	\[
	\sup_{N \in \mathbb{N}} \frac{1}{N} \sum_{n=1}^{N} \|T^n y\|
	\le C\|y\|
	\quad \text{for all } y \in Y.
	\]
	This property is closely related to mean Li--Yorke chaos in the context of Banach
	spaces. In fact, if an operator $T$ is mean Li--Yorke chaotic, then it is not
	absolutely Cesàro bounded; see, for instance,
	\cite{BermBMP, BBP} for further details. Observe that item (B) of Theorem \ref{ttt1} is equivalent to saying that \(B_w\) is not absolutely Ces\`aro bounded. Therefore, by specializing Theorem \ref{ttt1} to the spaces \(\ell^p(\mathbb{Z})\) (\(p \in [1,\infty)\)) or \(c_0(\mathbb{Z})\), we recover the characterization given in \cite{Nilson1}:
	\begin{corollary}\cite[Corollaries 73 and 78]{Nilson1}\label{f3}
		A weighted backward shift $B_w$ on $X:=\ell^p(\Z)$ ($p \in [1, \infty)$) or $X:=c_0(\mathbb{Z})$ with nonzero weights is mean Li--Yorke chaotic if and only if it is not absolutely Cesàro bounded and
		$$\liminf_{N \to \infty} \frac{1}{N} \sum_{n=1}^{N} |w_{-n} \cdots w_{-1}| = 0.$$
	\end{corollary}	
	\begin{remark}
		We note that the unilateral versions of Proposition \ref{prop}, Theorem \ref{ttt1} and Corollary \ref{f3} can be obtained in a straightforward way by using analogous arguments to those in the proofs of the bilateral versions.
	\end{remark}



	\smallskip
	
	{\footnotesize

		\bigskip\noindent
		{\sc Jo\~ao V. A. Pinto}
		
		\smallskip\noindent
		Departamento de Matem\'atica Aplicada, Instituto de Matem\'atica, Universidade Federal do Rio de Janeiro, 
		Caixa Postal 68530, RJ 21941-909, Brazil.\\
		\textit{ e-mail address}: joao.pinto@ufu.br
		
	}
	
\end{document}